\documentclass{elsarticle}
\usepackage{amsmath,amsfonts,amssymb,amsthm} 
\RequirePackage{slashed}

\input xy
\xyoption{all}

\newtheorem{thm}{Theorem}
\newtheorem{corl}[thm]{Corollary}
\newtheorem{lma}[thm]{Lemma} 
\newtheorem{prop}[thm]{Proposition}
\newtheorem{defn}[thm]{Definition}

\newtheorem{ex}[thm]{Example}
\newtheorem{rem}[thm]{Remark}
\newtheorem{ass}{Assumption}

\def\A{\mathcal{A}}

\def\B{\mathcal{B}}

\def\C{\mathbb{C}}

\def\dirac{\slashed{\partial}}

\def\eps{\varepsilon}

\def\H{\mathcal{H}}

\def\L{\mathcal{L}}

\def\R{\mathbb{R}}

\def\supp{\textup{supp}}

\DeclareMathOperator{\tr}{Tr}

\renewcommand{\phi}{\varphi}

\def\Z{\mathbb{Z}} 

\title{Perturbations and operator trace functions}
\author{Walter D. van Suijlekom}
\address{Institute for Mathematics, Astrophysics and Particle Physics,
Radboud University Nijmegen, Heyendaalseweg 135, 6525 AJ Nijmegen, The Netherlands; \\ e-mail: \texttt{waltervs@math.ru.nl}}
\date{\today, 2010}

\begin{document}
\bibliographystyle{plainmath}

\begin{abstract}
We study the spectral functional $A \mapsto \tr f(D+A)$ for a suitable function $f$, a self-adjoint operator $D$ having compact resolvent, and a certain class of bounded self-adjoint operators $A$. Such functionals were introduce by Chamseddine and Connes in the context of noncommutative geometry. Motivated by the physical applications of these functionals, we derive a Taylor expansion of them in terms of G\^ateaux derivatives. This involves divided differences of $f$ evaluated on the spectrum of $D$, as well as the matrix coefficients of $A$ in an eigenbasis of $D$. This generalizes earlier results to infinite dimensions and to any number of derivatives.
\end{abstract}
\maketitle

\section{Introduction}

The spectral action in noncommutative geometry \cite{C94} is given as the trace $\tr f(D)$ of a suitable function $f(D)$ of an unbounded self-adjoint operator $D$, which is assumed to have compact resolvent. One is interested in this trace function as $D$ is perturbed to $D+A$ where $A$ is a certain self-adjoint bounded operator. For instance, the so-called inner fluctuations of a spectral triple are of this type; they are central in the applications of noncommutative geometry to high-energy physics \cite{CC96,CC97,CCM07} (cf. also \cite{CM07}). A natural question that arises is what happens to the trace function when $D$ is perturbed to $D+A$. It is the goal of this paper to address this question.

We aim for a Taylor expansion of the spectral action by G\^ateaux deriving it with respect to $A$. As we will see, the context of finite dimensional noncommutative manifolds ({\it i.e.} spectral triples) allows for a derivation of results previously obtained only for finite dimensional (matrix) algebras \cite{Han06}. Our main result is the expansion:
 $$
S_D[A] = \sum_{n=0}^\infty \frac{1}{n} \sum_{i_1, \ldots i_n} A_{i_1 i_2}\cdots A_{i_n i_1} f' [\lambda_{i_1},\ldots,\lambda_{i_n}]
$$
where $f'[\lambda_{i_1},\ldots, \lambda_{i_n}]$ is the divided difference of order $n$ of $f'$ (cf. Defn \ref{defn:div-diff} below) evaluated on the spectrum of $D$, and $A_{ij}$ are the matrix coefficients with respect to an eigenbasis of $D$.

This paper is organized as follows. First, we recall in Section \ref{sect:pert} some results on perturbations of operators, in the setting of noncommutative geometry. Then, we give a precise definition of the spectral action functional in Section \ref{sect:trace}. In that section, we also recall the definition of divided differences and derive our main result on the Taylor expansion of the spectral action. We end with some conclusions and an appendix recalling a Theorem by Getzler and Szenes.

\section{Perturbations and spectral triples}
\label{sect:pert}

Recall that a {\it spectral triple} consists of an algebra $\A$ of bounded operators on a Hilbert space $\H$, together with a self-adjoint operator $D$ with compact resolvent such that the commutator $[D,a]$ is a bounded operator for all $a \in \A$. The key example is associated to a compact Riemannian spin manifold $M$: 
$$(C^\infty(M), L^2(M,S), \dirac),
$$ 
where $\dirac$ is a Dirac operator on the spinor bundle $S \to M$. Indeed, $\dirac$ is an elliptic differential operator of degree one and smooth functions satisfy
$$
\| [\dirac ,f] \| = \| f \|_{\textbf{Lip}} < \infty
$$
in the Lipschitz norm of $f$. 

In general, a spectral triple $(\A,\H,D)$ is said to be of {\it finite summability} if there exists an $n \geq 0$ such that $(1+D^2)^{-n/2}$ is a traceclass operator on $\H$. Let us start with a basic and well-known result.
\begin{lma}
\label{lma:traceheat}
Let $p$ be a polynomial on $\R$. Then for any $t>0$ the operator $p(D) e^{-t D^2}$ is traceclass.
\end{lma}
\begin{proof}
By finite summability and H\"olders inequality $(1+D^2)^{-n/2}$ is traceclass for some $n$. Thus, 
$$
p(D) e^{-tD^2} = \phi(D) (1+D^2)^{-n/2}
$$
with $\phi$ defined by functional calculus for the function
$$
\phi(x) = p(x) (1+x^2)^{-n/2} e^{-t x^2}.
$$
For $t>0$, this is a bounded function on $\R$ so that $\phi(D) (1+D^2)^{-n/2}$ is in the ideal $\L^1(\H)$ of traceclass operators as required.
\end{proof}
In particular, this applies to $p(x)=1$, {\it i.e.} finite summability implies so-called $\theta$-summability:
\begin{equation}
\label{eq:thetasumm}
\tr (e^{-tD^2}) < \infty, \qquad (t \in \R_+).
\end{equation}


\subsection{Fr\'echet algebra of smooth operators}
Given the derivation $\delta(\cdot) = [|D|,\cdot]$ on $\B(\H)$, there is a natural structure of a Fr\'echet algebra on the smooth domain of $\delta$. 
\begin{prop}
The following define a multiplicative family of semi-norms on $\B(\H)$:
$$
\| \delta^n(T) \| \qquad (T \in \B(\H)),
$$
indexed by $n \in \Z_{\geq 0}$. 
\end{prop} 
\begin{proof}
The derivation property of $\delta$ yields 
$$
 \| \delta^n(T_1 T_2) \| = \left\|  \sum_{k=0}^n {n \choose k} \delta^k (T_1) \delta^{n-k}(T_2) \right\|  \leq \sum_{k=0}^n {n \choose k}\| \delta_k (T_1) \|\|\delta_{n-k}(T_2)\| .
$$
\end{proof}
We will denote 
$$
\B^n(\H) = \left\{ T \in \B(\H): \| \delta^k(T)\| < \infty \text{ for all } k \leq n \right\}.
$$
Evidently, we have
$$
\B^\infty(\H) \subset \cdots \subset \B^2(\H) \subset \B^1(\H) \subset\B(\H)
$$
where by definition $\B^\infty(\H) = \cap_{n \in \Z_{\geq0}} \B^n(\H)$.
\begin{rem}
\label{rem:connes-1-forms}
Recall that a spectral triple $(\A, \H,D)$ is called {\rm regular} if both the algebra $\A$ and $[D,\A]$ are in the smooth domain of $\delta$. This can thus be reformulated as: 
$$
\text{the algebra generated by } a \text{ and }[D,b] ~ (a,b \in \A) \text{ is a subalgebra of } \B^\infty(\H).
$$
In particular, the $\A$-bimodule of {\it Connes' differential one-forms} \cite[Sect. VI.1]{C94},
$$
\Omega^1_D(\A) = \left\{ \sum_j a_j [D,b_j] \right\},
$$ 
is a subspace of $\B^\infty(\H)$.
\end{rem}

\subsection{Perturbations of heat operators}
In this subsection, we take a closer look at the heat operator $e^{-tD^2}$ and its perturbations. First, recall that the standard $m$-simplex is given by an $m$-tuple $(t_1,\ldots,t_m)$  satisfying $0 \leq t_1 \leq \ldots \leq t_m \leq 1$. Equivalently, it can be given by an $m+1$-tuple $(s_0, s_1, \ldots, s_m)$ such that $s_0+\ldots +s_m=1$ and $0\leq s_i \leq 1$ for any $i=0,\ldots,m$. Indeed, we have $s_0=t_1$, $s_i=t_{i+1}-t_i$ and $s_m=1-t_m$ and, vice versa, $t_k=s_0+s_1 + \cdots s_{k-1}$. 

For later use, we prove the following bound, which already appeared in a slightly different form in \cite{GS89}.
\begin{prop}
\label{prop:bounds-simplex}
For any $m \geq 0$ and $0 \leq k \leq m+1$ we have the bound
$$
\int_{\Delta_m} d^m s ( s_0 \cdots s_{k-1})^{-1/2} \leq \frac{\pi^k}{(m-k)!}.$$
\end{prop} 
\begin{proof}
In terms of the parameters $t_i$ for the $m$-simplex, we have to find an upper bound for
$$
\int^1_0 dt_m \int^{t_m}_0 dt_{m-1} \cdots \int^{t_2}_0 dt_1 \frac{1}{\sqrt{t_1 (t_2-t_1)\cdots (t_k-t_{k-1})}}
$$
where $t_{m+1} \equiv 1$.
First, note that by a standard substitution
$$
\int^{t_2}_0 dt_1 \frac{1}{\sqrt{t_1 (t_2-t_1)}} = \pi.
$$
For the subsequent integral over $t_2$:
$$
\int^{t_3}_0 dt_2 \frac{1}{\sqrt{t_3-t_2}} \leq \int^{t_3}_0 dt_2 \frac{1}{\sqrt{t_2(t_3-t_2)}} = \pi
$$
since $t_2 \leq 1$. This we can repeat $k$ times, leaving us with the integral
$$
\int^1_0 dt_m \int^{t_m}_0 dt_{m-1} \cdots \int^{t_{k+1}}_0 dt_k = \frac{1}{(m-k)!}.
$$
\end{proof}

\begin{lma}
\label{lma:heatoperator}
Let $A$ be a bounded operator and denote $D_A = D+A$. Then
$$
e^{-t (D_A)^2} = e^{-tD^2} - t \int_0^1 ds ~ e^{-st (D_A)^2} P(A) e^{-(1-s)t D^2}
$$
with $P(A)= DA+AD+A^2$.
\end{lma}
\begin{proof}
Note that $e^{-t D_A^2}$ is the unique solution of the Cauchy problem
$$
\left\{ \begin{array}{r} \left( d_t + D_A \right) u(t) = 0 \\ u(0)=1 \end{array} \right.
$$
with $d_t = d/dt$. Using the fundamental theorem of calculus, we find that 
\begin{multline*}
d_t \left[ e^{-tD^2} - \int_0^t dt' e^{-(t-t')D_A^2} P(A) e^{-t'D^2} \right]
\\
= -D_A^2 \left( e^{-tD^2} -  \int_0^t dt' e^{-(t-t')D_A^2} P(A) e^{-t'D^2}  \right)
\end{multline*}
showing that the bounded operator $e^{-tD^2} - \int_0^t dt' e^{-(t-t')D_A^2} P(A) e^{-t'D^2}$ also solves the above Cauchy problem. 
\end{proof}

The following estimates were proved in a slightly different form in \cite{GS89}. 
\begin{lma}

\label{lma:estimates}
If the operators $A, A_i$ are bounded, and $\alpha_i \in \{0,1\}$ are such that $\sum_i \alpha_i =k$, then
\begin{multline*}
\left | \int_{\Delta_n}  \tr A_0 |D_A|^{\alpha_0} e^{-s_0 t D_A^2} 
A_1 |D|^{\alpha_1} e^{-s_1 t D^2} \cdots A_n |D|^{\alpha_n} e^{-s_n t D^2} 
d^ns 
\right| \\
\leq \frac{  
\| A_0 \| \cdots \|A_n\| \tr e^{-(1-\epsilon) t D^2}}{(n-k)!(\pi^{-2}\epsilon t)^{k/2}} 
\end{multline*}
for any $0 < \epsilon <1$.
\end{lma}
\begin{proof}
Recall H\"older's inequality:
\begin{equation}
\label{eq:holder}
\left| \tr (T_0 \cdots T_n) \right| \leq \| T_0 \|_{s_0^{-1}} \cdots \| T_n \|_{s_n^{-1}}
\end{equation}
when $s_0 + \cdots + s_n =1$. Also, we estimate for some arbitrary $0< \epsilon < 1$
\begin{align*}
\left\| A_i e^{-s_i t D^2} \right\|_{s_i^{-1}}& \leq \| A_i \| \left(\tr e^{-t D^2}\right)^{s_i} \leq \| A_i \| \left(\tr e^{-(1-\epsilon) t D^2}\right)^{s_i}\\
\left\| A_i |D| e^{-s_i t D^2} \right\|_{s_i^{-1}} & \leq \| A_i \| \left\| |D| e^{-\epsilon s_i t D^2 }\right\|  \left(\tr e^{-(1-\epsilon)t D^2}\right)^{s_i} \\
 &\leq (\epsilon s_i t)^{-1/2} \| A_i \| \left(\tr e^{-(1-\epsilon) t D^2}\right)^{s_i}
\end{align*}
writing $e^{-st D^2} = e^{-\epsilon st D^2} e^{-(1-\epsilon)st D^2} $. We have used Lemma \ref{lma:traceheat} and the fact that
$$
\left\|  e^{-\epsilon s tD^2} \right\| \leq 1; \qquad 
\left\| |D|   e^{-\epsilon s tD^2}\right \| \leq  \sup_{x \in \R_+} \{ xe^{-\epsilon st x^2} \} =(2e \epsilon st)^{-1/2}.
$$
Moreover, Theorem C in \cite{GS89} (cf. Appendix \ref{sect:theoremC}) gives
\begin{equation}
\label{thmC}
\tag{$\ast$}
 \tr e^{-t (1-\epsilon/2) (D_A)^2} \leq e^{(1+ 2/\epsilon) t \| A\|^2}  \tr e^{-t (1-\epsilon) D^2}.
\end{equation}
This further yields
\begin{align*}
\left\| A_0 |D_A| e^{-s_0 t D_A^2} \right\|_{s_0^{-1}}  &\leq \| A_0 \| \left\| |D_A| e^{-\epsilon/2 s_i t D_A^2 }\right\|  \left(\tr e^{-(1-\epsilon/2)t D_A^2}\right)^{s_i} \\
&\leq (e \epsilon s_0 t)^{-1/2} e^{(1+2/\epsilon)t \|A\|^2} \| A_0\| \left(\tr e^{-(1-\epsilon) t D^2}\right)^{s_0}.
\end{align*}
Combining these estimates with \eqref{eq:holder}, we obtain for instance in the case that the first $k$ $\alpha_i$ are nonzero ({\it i.e.} $\alpha_0= \cdots =\alpha_{k-1}=1$):
\begin{multline*}
\left | \tr A_0 |D_A|^{\alpha_0} e^{-s_0 t D_A^2} A_1 |D|^{\alpha_1} e^{-s_1 t D^2} \cdots A_n |D|^{\alpha_n} e^{-s_n t D^2}  
\right| \\
\leq \frac{  \| A_0 \| \cdots \|A_n\|}{s_0 \cdots s_k (\epsilon t)^{k/2}} \tr e^{-(1-\epsilon) t D^2}
\end{multline*}
making use of the fact that $s_0+s_1+ \cdots s_n=1$. The bounds of Proposition \ref{prop:bounds-simplex} complete the proof.
\end{proof}

Let us introduce the following convenient notation (cf. \cite{GS89}). If $A_0, \ldots, A_n$ are operators, we define a $t$-dependent quantity by
\begin{equation}
\label{brackets}
\left\langle A_0, \ldots, A_n \right\rangle_n := t^{n} \tr \int_{\Delta_n} 
A_0 e^{-s_0tD^2} A_1 e^{-s_1tD^2} \cdots A_n e^{-s_ntD^2}  d^ns.
\end{equation}
Note the difference in notation with \cite{GS89}, for which the same symbol is used for the supertrace of the same expression, rather than the trace. Also, we are integrating over the `inflated' $n$-simplex $t \Delta^n$, yielding the factor $t^n$. The forms $\langle A_0, \ldots, A_n\rangle$ satisfy, {\it mutatis mutandis}, the following properties.
\begin{lma}{\cite{GS89}}
In each of the following cases, we assume that the operators $A_i$ are such that each term is well-defined.
\begin{enumerate}
\item $\langle A_0, \ldots, A_n \rangle_n = \langle A_i, \ldots, A_n,\ldots, A_{i-1} \rangle_n $
\item $\langle A_0, \ldots, A_n \rangle_n = \sum_{i=0}^n \langle 1, \ldots, A_i ,\ldots, A_n, A_0, \ldots, A_{i-1} \rangle_n $
\item $\sum_{i=0}^n \langle A_0, \ldots,[D,A_i],\ldots A_n \rangle_n=0$
\item $\langle A_0, \ldots,[D^2,A_i], \ldots, A_n \rangle_n = \langle A_0, \ldots, A_{i-1}A_i,,\ldots, A_n \rangle_{n-1} \\$
$ { } \qquad \qquad \qquad \qquad \qquad\quad \quad- \langle A_0, \ldots,A_i A_{i+1},\ldots, A_n \rangle_{n-1}$
\end{enumerate}
\end{lma}

\subsection{G\^ateaux derivatives}
As a preparation for the next section, we recall the notion of G\^ateaux derivatives, referring to the excellent treatment \cite{Ham82} for more details.

\begin{defn}
The {\it G\^ateaux derivative} at $x \in X$ of a map $F: X \to Y$ between locally convex topological vector spaces is defined for $h \in X$ by
\begin{align*}
F'(x)(h) = \lim_{u \to 0} \frac{F(x+uh) - F(x)}{u}.
\end{align*}
\end{defn}
In general, the map $F'(x)(\cdot)$ is not linear, in contrast with the Fr\'echet derivative. However, if $X$ and $Y$ are Fr\'echet spaces, then the G\^ateaux derivatives actually defines a linear map $F'(x)(\cdot)$ for any $x \in X$ \cite[Theorem 3.2.5]{Ham82}. In this case, higher order derivatives are denoted as $F'', F'''$ {\it et cetera}, or more conveniently as $F^{(k)}$ for the $k$-th order derivative. The latter will be understood as a linear bounded operator from $X \times \cdots \times X$ ($k+1$ copies) to $Y$.
\begin{thm}[Taylor's formula with integral remainder]
\label{thm:taylor}
For a G\^ateaux $k+1$-differentiable map $F: X\to Y$ between Fr\'echet spaces $X$ and $Y$ it holds for $x, a \in X$ that
\begin{multline*}
F(x) = F(a) + F'(a)(x-a) + \frac{1}{2!} F''(a)(x-a, x-a) + \cdots \\
 + \frac{1}{n!} F^{(k)} (a)(x-a, \ldots, x-a) + R_k(x)
\end{multline*}
with integral remainder given by 
$$
R_k(x) = \frac{1}{k!} \int_0^1 F^{(k+1)}(a+t(x-a)) ((1-t)h, \ldots, (1-t)h, h) dt.
$$
\end{thm}

\section{Trace functionals}
\label{sect:trace}
In this section, we consider trace functionals of the form $A \mapsto\tr f(D+A)$.
Here $D$ is the self-adjoint operator forming a finitely summable spectral triple $(\A,\H,D)$, and $A$ is a bounded operator. We derive a Taylor expansion of this functional in $A$.
Our main motivation comes from the spectral action principle introduced by Chamseddine and Connes \cite{CC96,CC97} and we define accordingly
\begin{defn}[Chamseddine--Connes \cite{CC97}]
The {\rm spectral action functional} $S_D[A]$ is defined by
$$
S_D[A] = \tr f\left( D+A \right) ; \qquad (A \in \B(\H)).
$$
\end{defn}
\noindent The square brackets indicate that $S_D[A]$ is considered as a functional of $A \in \B(\H)$. 

\begin{rem}
Actually, Chamseddine and Connes considered $S_D[A]$ for so-called gauge fields associated to the spectral triple $(\A,\H,D)$. These are self-adjoint elements $A$ in $\Omega^1_D(\A)$ which by Remark \ref{rem:connes-1-forms} is a subset of $\B^2(\H)$. 
\end{rem}

For the function $f$ we assume that it is a Laplace--Stieltjes transform:
$$
f(x) = \int_{t >0} e^{-t x^2}d \mu(t)
$$
for which we make the additional
\begin{ass}
\label{ass:laplace}
For all $\alpha>0, \beta>0, \gamma>0$ and $0 \leq \epsilon<1$, there exist constants $C_{\alpha\beta\gamma\epsilon}$ such that
$$
 \int_{t >0} \tr t^\alpha |D|^\beta e^{-t (\epsilon D^2 - \beta)} \left| d \mu(t) \right| < C_{\alpha\beta\gamma\epsilon}.
$$
\end{ass}

In view of Theorem \ref{thm:taylor}, we have the following Taylor expansion (around 0) in $A \in \B^2(\H)$ for the spectral action $S_D[A]$:
\begin{equation}
\label{eq:taylor-action}
S_D[A] = \sum_{n=0}^\infty \frac{1}{n!} S_D^{(n)}(0)(A, \ldots, A).
\end{equation}
Indeed, $S_D$ is Fr\'echet differentiable on $\B^2(\H)$ as the following Proposition establishes.
\begin{prop}
\label{prop:sa-n-der}
If $n=0,1,\ldots$ and $A \in \B^2(\H)$, then $S_D^{(n)}(0)(A,\ldots,A)$ exists and
\begin{multline*}
S_D^{(n)} (0)(A,\ldots, A) =n! \sum_{k=0}^n (-1)^k  \sum_{\eps_1, \ldots, \eps_k} \langle 1, (1-\eps_1) \{D,A\}+ \eps_1 A^2, \ldots,  \\
(1-\eps_k) \{D,A\}+ \eps_k A^2 \rangle_k ~ d\mu(t),
\end{multline*}
where the sum is over multi-indices $(\eps_1, \ldots, \eps_k) \in \{0,1\}^k$ such that $\sum_{i=1}^k (1+\eps_i) = n$.
\end{prop}
\begin{proof}
We will prove this by induction on $n$; the case $n=0$ being trivial.
By definition of the G\^ateaux derivative and using Lemma \ref{lma:heatoperator}
\begin{align*}
S^{(n+1)}_D(0)(A,\ldots, A) &= n! \sum_{k=0}^n  \sum_{\eps_1, \ldots, \eps_k} \Bigg[
\sum_{i=1}^k  (-1)^{k+1} \langle 1, (1-\eps_1) \{D,A\}+ \eps_1 A^2,\\
&\qquad \qquad \qquad  \ldots, \underset{i}{\{ D, A\}}, \ldots, (1-\eps_k) \{D,A\}+ \eps_k A^2 \rangle_{k+1} \\
&\quad + \sum_{i=1}^k  (-1)^{k} \langle 1, (1-\eps_1) \{D,A\}+ \eps_1 A^2, \ldots, 2(1-\eps_i)A^2, \\
& \qquad \qquad \qquad \qquad  \ldots, (1-\eps_k) \{D,A\}+ \eps_k A^2 \rangle_{k} \Bigg] ~d\mu(t).
\end{align*}
The first sum corresponds to a multi-index $\vec{\eps}~'= (\eps_1, \ldots, \eps_{i-1},0,\eps_{i}, \ldots, \eps_k)$, the second sum corresponds to $\vec{\eps}~'= (\eps_1, \ldots, \eps_{i}+1, \ldots, \eps_k)$ if $\eps_i  = 0$, counted with a factor of $2$. In both cases, we compute that $\sum_j (1+\eps'_j) = n+1$. In other words, the induction step from $n$ to $n+1$ corresponds to inserting in a sequence of $0$'s and $1$'s (of, say, length $k$) either a zero at any of the $k+1$ places, or replace a $0$ by a $1$ (with the latter counted twice). In order to arrive at the right combinatorial coefficient $(n+1)!$, we have to show that any $\vec{\eps}~'$ satisfying $\sum_i (1+\eps_i')= n+1$ appears in precisely $n+1$ ways from $\vec{\eps}$ that satisfy $\sum_i (1+\eps_i)= n$. If $\vec{\eps}~'$ has length $k$, it contains $n+1-k$ times $1$ as an entry and, consequently, $2k-n-1$ a $0$. This gives (with the double counting for the $1$'s) for the number of possible $\vec\eps$:
$$
2(n+1-k) + 2k-n-1 = n+1
$$
as claimed. This completes the proof.
\end{proof}

\begin{ex}
\begin{align*}
S_D^{(1)}(0)(A)&= \int \bigg(- \langle 1, \{D,A\} \rangle_1\bigg) ~d\mu(t) \\
S_D^{(2)}(0)(A,A)&= 2\int\bigg( -\langle1,  A^2 \rangle_1 + \langle 1, \{D, A\} , \{D,A\} \rangle_2 \bigg) ~d\mu(t) \\
S_D^{(3)}(0)(A,A,A)&= 3! \int\bigg( \langle1,  A^2,\{D, A\} \rangle_2+ \langle1,  \{D, A\},A^2 \rangle_2 \\
& \qquad  \qquad \qquad - \langle 1, \{D, A\} , \{D,A\},\{D,A\} \rangle_3 \bigg) ~d\mu(t) \\
\end{align*}
\end{ex}

\subsection{Divided differences}
Recall the definition of and some basic results on divided differences.
\begin{defn} 
\label{defn:div-diff}
Let $g: \R \to \R$ and $x_0, x_1, \ldots x_n$ be distinct points on $\R$. The {\rm divided difference of order $n$} is defined by the recursive relations
\begin{align*}
g[x_0] &= g(x_0), \\
g[x_0,x_1, \ldots x_n] &= \frac{ g[x_1, \ldots x_n] -g[x_0,x_1, \ldots x_{n-1}]}{x_{n} - x_0}.
\end{align*}
On coinciding points we extend this definition as the usual derivative:
$$
g[x_0, \ldots,x \ldots, x \ldots x_n]:= \lim_{u \to 0} g[x_0, \ldots,x+u \ldots, x \ldots x_n]
$$
Finally, as a shorthand notation, we write for an index set $I =\{i_1, \ldots, i_n \}$:
$$
g[x_{I}] = g[x_{i_1}, \ldots, x_{i_n}]. 
$$
\end{defn}
Also note the following useful representation due to Hermite \cite{Her1878}. 
\begin{prop}
\label{hermite}
For any $x_0, \ldots, x_n \in \R$
$$
f[x_0, x_1, \ldots, x_n] = \int_{\Delta_n} f^{(n)} \left(s_0 x_0 + s_1 x_1 + \cdots + s_n x_n\right) d^ns.
$$
\end{prop}
As an easy consequence, we derive
$$
\sum_{i=0}^n f[x_0, \ldots, x_i,x_i, \ldots, x_n]
= f'[x_0, x_1, \ldots, x_n].
$$

\begin{prop}
\label{prop:div-diff:f-g}
For any $x_1, \ldots x_n \in \R$ we have for $f(x)= g(x^2)$:
\begin{align*}
f[x_0,\cdots,x_n]&= \sum_I  \left( \prod_{\{ i-1, i \} \subset I} (x_i+x_{i+1}) \right) g[x_I^2] 
\end{align*}
where the sum is over all ordered index sets $I= \{ 0 = i_0 < i_1<\ldots<i_k =n \}$ such that $i_{j} - i_{j-1} \leq 2$ for all $1 \leq j \leq k$ ({\it i.e.} there are no gaps in $I$ of length greater than 1).
\end{prop}
\begin{proof}
This follows from the chain rule for divided difference: if $f = g \circ \phi$, then \cite{FL07}
$$
f[x_0, \ldots x_n] = \sum_{k=1}^n \sum_{ 0 = i_0 < i_1<\ldots<i_k =n } g[\phi(x_{i_0}), \ldots, \phi(x_{i_k})] \prod_{j=0}^{k-1} \phi[ x_{i_j}, \ldots, x_{i_{j+1}}].
$$
For $\phi(x) = x^2$ we have $\phi[x,y]= x+y$, $\phi[x,y,z]= 1$ and all higher divided differences are zero. Thus, if $i_{j+1} - i_j > 2$ then $\phi[ x_{i_j}, \ldots, x_{i_{j+1}}]=0$. In the remaining cases one has
$$
\phi[ x_{i_j}, \ldots, x_{i_{j+1}}]= \left\{ \begin{array}{lll}
x_{i_j}+ x_{i_{j+1}} &\text{if} &i_{j+1} - i_j =1 \\
1 &\text{if}& i_{j+1} - i_j =2 \\
\end{array}\right.
$$
and this selects in the above summation precisely the index sets $I$.
\end{proof}
\begin{ex}
For the first few terms, we have
\begin{align*}
f[x_0,x_1] &= (x_0+x_1) g[x_0^2, x_1^2]\\
f[x_0,x_1,x_2] &= (x_0+x_1)(x_1+x_2) g[x_0^2, x_1^2,x_2^2] + g[x_0^2,x_2^2] \\
f[x_0,x_1,x_2,x_3] &= (x_0+x_1)(x_1+x_2)(x_2+x_3) g[x_0^2,x_1^2, x_2^2,x_3^2]
\\ &\quad 
+ (x_2+x_3) g[x_0^2,x_2^2,x_3^2] + (x_0+x_1) g[x_0^2,x_1^2, x_3^2]
\end{align*}
\end{ex}

\subsection{Taylor expansion of the spectral action}
We fix a complete set of eigenvectors $\{ \psi_n \}_n$ of $D$ with respective eigenvalue $\lambda_n \in \R$, forming an orthonormal basis for $\H$. We also denote $A_{mn} := (\psi_m , A \psi_n)$ so that $\sum_{m,n} A_{mn} | \psi_m )( \psi_n |$ converges to $A$ in the weak operator topology. 
\begin{thm}
\label{thm:sa-n-der-dd}
If $f$ satisfies Assumption \ref{ass:laplace} and $A \in \B^2(\H)$, then 
$$
S_D^{(n)}(0)(A,\ldots,A) = n! \sum_{i_1, \ldots, i_n} A_{i_n i_1} A_{i_1 i_2} \cdots A_{i_{n-1} i_n} f [ \lambda_{i_p}, \lambda_{i_1}, \ldots, \lambda_{i_n}].
$$
\end{thm}
A similar result was obtained in finite dimensions in \cite{Han06}. 
\begin{proof}
Proposition \ref{prop:sa-n-der} gives us an expression for $S_D^{(n)}$ in terms of the brackets $\langle \cdots \rangle$. We compute for these:
\begin{align*}
&(-1)^k\langle 1, (1-\eps_1) \{D,A\}+ \eps_1 A^2, \ldots, (1-\eps_k) \{D,A\}+ \eps_k A^2 \rangle_k ~d\mu(t) \\
& \qquad =(-1)^k \!\!\!\! \!\!\!\!  \sum_{i_0=i_k, i_1,\ldots, i_k} \int_{\Delta_k} \left( \prod_{j=1}^k \left((1-\eps_j) (\lambda_{i_{j-1}} - \lambda_{i_j}) A+ \eps_j A^2 \right)_{i_{j-1} i_j} \right) \\
  & \qquad \qquad \qquad \qquad \qquad \qquad \qquad \qquad \times e^{-( s_0t \lambda_{i_0}^2 + \cdots + s_k t\lambda_{i_k}^2)} d^k s d\mu(t) \\
& \qquad = 
 \sum_{i_0=i_k, i_1,\ldots, i_k} \left( \prod_{j=1}^k \left((1-\eps_j) (\lambda_{i_{j-1}} - \lambda_{i_j}) A+ \eps_j A^2 \right)_{i_{j-1} i_j} \right) g[\lambda_{i_0}^2, \ldots, \lambda_{i_k}^2].
\end{align*}
Glancing back at Proposition \ref{prop:div-diff:f-g} we are finished if we establish a one-to-one relation between the order index sets $I=\{ 0 = i_0 < i_1 < \cdots < i_k = n\}$ such that $i_{j-1} - i_j \leq 2$ for all $1 \leq j \leq k$ and the multi-indices $(\eps_1, \ldots, \eps_k) \in \{ 0,1\}^k$ such that $\sum_{i=1}^k (1+\eps_i) = n$. If $I$ is such an index set, we define a multi-index:
$$
\eps_j = \left\{ \begin{array}{ll} 0 &\text{if } \{ i_j-1,i_j\} \subset I ,\\
1 &\text{otherwise.} \\ \end{array} \right.
$$
Indeed, then $i_j = i_{j-1} +1 + \eps_j$ so that 
$$
\sum_{i=1}^k (1+\eps_i) = i_0 + \sum_{i=1}^k (1+\eps_i) = i_k = n.
$$
It is now clear that, vice-versa, if $\eps$ is as above, we define $I= \{ 0 = i_0 < i_1 < \cdots < i_k = n\}$ by $i_j = i_{j-1}+1 + \eps_j$ and starting with $i_0=0$.
\end{proof}

\begin{corl}
If $n\geq 0$ and $A \in \B^2(\A)$, then
$$
S_D^{(n)}(0)(A,\ldots,A) = (n-1)!\sum_{i_1, \ldots i_n} A_{i_1 i_2}\cdots A_{i_n i_1} f' [\lambda_{i_1},\ldots,\lambda_{i_n}].
$$
Consequently, 
$$
S_D[A] = \sum_{n=0}^\infty \frac{1}{n} \sum_{i_1, \ldots i_n} A_{i_1 i_2}\cdots A_{i_n i_1} f' [\lambda_{i_1},\ldots,\lambda_{i_n}].
$$
\end{corl}

An interesting consequence is the following, which was obtained recently at first order for bounded operators \cite{GHJR09}. 
\begin{corl}
If $n\geq 0$ and $A \in \B^2(\A)$ and if $f'$ has compact support, then
$$
S_D^{(n)}(0)(A,\ldots,A) = \frac{(n-1)!}{2\pi i} \tr \oint f'(z)A  (z-D)^{-1} \cdots A  (z-D)^{-1}.
$$
The contour integral encloses the intersection of the spectrum of $D$ with $\supp f'$.
\end{corl}
\begin{proof}
This follows directly from Cauchy's formula for divided differences (cf. \cite[Ch. I.1]{Don74})
$$
g[x_0, \ldots x_n] = \frac{1}{2\pi i} \oint \frac{g(z) }{(z-x_0)\cdots (z-x_n)} dz
$$
with the contour enclosing the points $x_i$. 
\end{proof}

\section{Outlook}

We have obtained a Taylor expansion for the spectral action in noncommutative geometry. As such, it is natural to consider its quadratic part as the starting point for a free quantum field theory. Expectedly, this involves the usual nuances of a gauge theory such as gauge fixing, Gribov ambiguities, {\it et cetera}. Under the assumption of vanishing tadpole 
$$
S_D^{(1)}(A) = 0; \qquad (A \in \Omega^1(\A)), 
$$
also exploited in \cite{CC06}, one indeed encounters a degeneracy in the quadratic part. In fact, in this case $S_D^{(2)}(A,[D,a]) = 0$ for all $a \in \A$. This vanishing on pure gauge fields will be considered in more detail elsewhere. Once this issue has been dealt with, the higher derivatives of the spectral action account for interactions, allowing for a development of a perturbative quantization of the spectral action. 

Another application of the present work is to matrix models, as our Taylor expansion is very similar to Lagrangians encountered in matrix models. In fact, if the spectral triple is $(M_N(\C), \C^N, D)$ with $D$ a symmetric $N \times N$-matrix, then the spectral action is exactly the hermitian one-matrix model (cf. \cite{FGZ95}). An honest infinite-dimensional example might be provided by the spectral triples that are involved in Moyal deformations (see \cite{Wul06} and references therein). It would be interesting to apply the above results and develop a quantum theory for these models.

\subsection*{Acknowledgements}
The author would like to thank Alain Connes and Dirk Kreimer for stimulating discussions. The Institut de Hautes \'Etudes Scientifique in Bures-sur-Yvette is thanked for providing a great scientific atmosphere during visits in 2009 and 2010. This work is part of the NWO VENI-project 639.031.827.

\appendix

\section{A theorem by Getzler and Szenes}
\label{sect:theoremC}
In \cite{GS89} Getzler and Szenes proof the following theorem. For completeness, we repeat it here (specified to our finitely-summable case).
\begin{thm}[Getzler-Szenes]
Let $(\A,\H,D)$ be a finitely-summable spectral triple and $V$ a selfadjoint bounded operator on $\H$. Then $(\A,\H, D_V)$ with $D_V = D+V$ is a finitely-summable spectral triple, and
$$
\tr e^{-(1-\epsilon/2) t (D_V)^2} \leq e^{(1+2/\epsilon) t\|V\|^2 } \tr e^{-(1-\epsilon)t D^2}
$$
for any $0< \epsilon<1$ and $t>0$.
\end{thm}
\begin{proof}
This follows from the fact that for two positive self-adjoint operator $A$ and $B$ we have
\begin{equation}
\label{eq:est-AB}
\tr e^{-A-B} \leq \tr e^{-A}.
\end{equation}
Indeed, let
\begin{align*}
A&= (1-\epsilon)t D^2\\
B&= \epsilon t D^2 /2 + (1-\epsilon/2) t (D V + V D + V^2) + (1+2/\epsilon)t \| V\|^2
\end{align*}
so that $A+B = (1-\epsilon/2) (D+V)^2 + (1+2/\epsilon) \| V\|^2$. Obviously, $A$ is positive. To see that $B$ is positive, we use the fact that
$$
0 \leq \epsilon t D^2/2 + 2 t V^2/\epsilon+t (DV+VD),
$$
which is just positivity of $(\sqrt{\epsilon t/2} D + \sqrt{2 t /\epsilon} V)^2$. Combining this with $V^2 \leq \| V^2 \|$ and multiplying by the positive number $(1-\epsilon/2)$ we obtain
$$
0 \leq (1-\epsilon/2) \left(\epsilon t D^2/2 + 2 t \| V\|^2 /\epsilon+ t (DV+VD) \right)
= B- \epsilon^2/4 t D^2 - (1-\epsilon/2)t V^2,
$$
ensuring positivity of $B$. Equation \eqref{eq:est-AB} then implies
$$
\tr e^{-(1-\epsilon/2) t (D^2 + DV+VD+V^2)} e^{-(1+2/\epsilon) t \| V\|^2}\leq  \tr e^{-(1-\epsilon)t D^2}
$$
as desired.
\end{proof}

\newcommand{\noopsort}[1]{}\def\cprime{$'$}

\end{document}